\documentclass[english]{amsart}
\usepackage[T1]{fontenc}
\usepackage[latin9]{inputenc}
\usepackage{amsthm}
\usepackage{amsmath}
\usepackage{amssymb}
\usepackage{esint}

\makeatletter
\theoremstyle{plain}
\newtheorem{thm}{\protect\theoremname}
  \theoremstyle{definition}
  \newtheorem{defn}[thm]{\protect\definitionname}
  \theoremstyle{plain}
  \newtheorem{lem}[thm]{\protect\lemmaname}
  \theoremstyle{remark}
  \newtheorem{rem}[thm]{\protect\remarkname}
  \theoremstyle{plain}
  \newtheorem{prop}[thm]{\protect\propositionname}

\DeclareMathOperator{\Id}{\text{Id}}

\makeatother

\usepackage{babel}
  \providecommand{\definitionname}{Definition}
  \providecommand{\lemmaname}{Lemma}
  \providecommand{\propositionname}{Proposition}
  \providecommand{\remarkname}{Remark}
\providecommand{\theoremname}{Theorem}

\begin{document}

\title[Nonlinear KGM systems]{Nonlinear Klein-Gordon-Maxwell systems with Neumann boundary conditions
on a Riemannian manifold with boundary}

\thanks{The first author was  supported by Gruppo Nazionale per l'Analisi Matematica, la Probabilit\`{a} e le loro 
Applicazioni (GNAMPA) of Istituto Nazionale di Alta Matematica (INdAM)}

\author{Marco Ghimenti}
\address{M. Ghimenti, \newline Dipartimento di Matematica Universit\`a di Pisa
Largo B. Pontecorvo 5, 56100 Pisa, Italy}
\email{marco.ghimenti@dma.unipi.it}

\author{Anna Maria Micheletti}
\address{A. M. Micheletti, \newline Dipartimento di Matematica Universit\`a di Pisa
Largo B. Pontecorvo 5, 56100 Pisa, Italy}
\email{a.micheletti@dma.unipi.it.}

\keywords{Riemannian manifold with boundary, Klein Gordon Maxwell systems, Neumann boundary condition, Mean curvature, Liapounov Schmidt}

\subjclass[2010]{35J57,35J60,58E05,81V10}

\begin{abstract}
Let $(M,g)$ be a smooth compact, $n$ dimensional Riemannian manifold,
$n=3,4$ with smooth $n-1$ dimensional boundary $\partial M$ . We
search the positive solutions of the singularly perturbed Klein Gordon
Maxwell Proca system with homogeneous Neumann boundary conditions
or for the singularly perturbed Klein Gordon Maxwell system with mixed
Dirichlet Neumann homogeneous boundary conditions. We prove that $C^{1}$
stable critical points of the mean curvature of the boundary generates
$H^{1}(M)$ solutions when the perturbation parameter $\varepsilon$
is sufficiently small.
\end{abstract}

\maketitle

\section{Introduction}

Let $(M,g)$ be a smooth compact, $n$ dimensional Riemannian manifold,
$n=3,4$ with boundary $\partial M$ which is the union of a finite
number of connected, smooth, boundaryless, $n-1$ submanifolds embedded
in $M$. Here $g$ denotes the Riemannian metric tensor. By Nash theorem
we can consider $(M,g)$ as a regular submanifold embedded in $\mathbb{R}^{N}$. 

We search the positive solutions of the following Klein Gordon Maxwell
Proca system with homogeneous Neumann boundary conditions 
\begin{equation}
\left\{ \begin{array}{cc}
-\varepsilon^{2}\Delta_{g}u+au=|u|^{p-2}u+\omega^{2}(qv-1)^{2}u & \text{in }M\\
-\Delta_{g}v+(1+q^{2}u^{2})v=qu^{2} & \text{in }M\\
\frac{\partial u}{\partial\nu}=\frac{\partial v}{\partial\nu}=0 & \text{ on }\partial M
\end{array}\right.\label{eq:Pnn}
\end{equation}
or Klein Gordon Maxwell system with mixed Dirichlet Neumann homogeneous
boundary conditions 
\begin{equation}
\left\{ \begin{array}{cc}
-\varepsilon^{2}\Delta_{g}u+au=|u|^{p-2}u+\omega^{2}(qv-1)^{2}u & \text{in }M\\
-\Delta_{g}v+q^{2}u^{2}v=qu^{2} & \text{in }M\\
v=0 & \text{ on }\partial M\\
\frac{\partial u}{\partial\nu}=0. & \text{ on }\partial M
\end{array}\right.\label{eq:Pdn}
\end{equation}
Here $2<p<2^{*}=\frac{2n}{n-2}$, $\nu$ is the external normal to
$\partial M$, $a>0$, $q>0$, $\omega\in(-\sqrt{a},\sqrt{a})$ and
$\varepsilon$ is a positive perturbation parameter. 

We are interested in finding solutions $u,v\in H_{g}^{1}(M)$ to problem
(\ref{eq:Pnn}) and\textbf{ }(\ref{eq:Pdn}). Also, we show that,
for $\varepsilon$ sufficiently small, the function $u$ has a peak
near a stable critical point of the mean curvature of the boundary.
\begin{defn}
\label{def:stable}Let $f\in C^{1}(N,\mathbb{R})$, where $(N,g)$
is a Riemannian manifold. We say that $K\subset N$ is a $C^{1}$-stable
critical set of $f$ if $K\subset\left\{ x\in N\ :\ \nabla_{g}f(x)=0\right\} $
and for any $\mu>0$ there exists $\delta>0$ such that, if $h\in C^{1}(N,\mathbb{R})$
with 
\[
\max_{d_{g}(x,K)\le\mu}|f(x)-h(x)|+|\nabla_{g}f(x)-\nabla_{g}h(x)|\le\delta,
\]
then $h$ has a critical point $x_{0}$ with $d_{g}(x_{0},K)\le\mu$.
Here $d_{g}$ denotes the geodesic distance associated to the Riemannian
metric $g$.

Now we state the main theorem.\end{defn}
\begin{thm}
\label{thm:main}Assume $K\subset\partial M$ is a $C^{1}$-stable
critical set of the mean curvature of the boundary. Then there exists
$\varepsilon_{0}>0$ such that, for any $\varepsilon\in(0,\varepsilon_{0})$,
Problem (\ref{eq:Pnn}) has a solution $(u_{\varepsilon},v_{\varepsilon})\in H_{g}^{1}(M)\times H_{g}^{1}(M)$.
Analogously, problem (\ref{eq:Pdn}) has a solution $(u_{\varepsilon},v_{\varepsilon})\in H_{g}^{1}(M)\times H_{0,g}^{1}(M)$.
Moreover, the function $u_{\varepsilon}$ has a peak in some $\xi_{\varepsilon}\in\partial M$
which converges to a point $\xi_{0}\in K$ as $\varepsilon$ goes
to zero.
\end{thm}
From the seminal paper of \cite{BF} many authors studied KGM systems
on a flat domain. We cite \cite{AP2,C,DW2,DM2,DPS2,DP,DPS1,Mu} .

For KGM and KGMP system on Riemannian manifolds, as far as we know
the first paper in which deals with this problem is by Druet and Hebey
\cite{DH}. In this work the authors study the case $\varepsilon=1$
and prove the existence of a solution for KGMP systems on a closed
manifold, by the mountain pass theorem. Thereafter several works are
devoted to the study of KGMP system on Riemaniann closed manifold.
We limit ourself to cite \cite{HT,HW,CGMkg,GMkg,GMP}. 

Klein Gordon Maxwell system provides a model for a particle $u$ interacting
with its own electrostatic field $v$. Thus, is somewhat more natural
to prescribe Neumann condition on the second equation as d'Avenia
Pisani and Siciliano nicely explained in the introduction of \cite{DPS2}. 

So, recently we moved to study KGMP systems in a Riemaniann manifold
$M$ with boundary $\partial M$ with Neumann boundary condition on
the second equation. In \cite{GMbordopreprint} the authors proved
that the topological properties of the boundary $\partial M$, namely
the Lusternik Schnirelmann category of the boundary, affects the number
of the low energy solution for the systems. Also, we notice that the
natural dimension for KGM and KGMP systems is $n=3$, since this systems
arises from a physical model. However, the case $n=4$ is interesting
from a mathematical point of view, since the second equation of systems
(\ref{eq:Pnn}) and (\ref{eq:Pdn}) becomes energy critical by the
presence of the $u^{2}v$ term. For further comments on this subject,
we refer to \cite{HT}

We can compare \cite{GMbordopreprint} and Theorem \ref{thm:main}.
In \cite{GMjfpta} we proved that the set of metrics for which the
mean curvature has only nondegenerate critical points is an open dense
set among all the $C^{k}$ metrics on $M$, $k\ge3$. Thus, generically
with respect to the metric, the mean curvature has $P_{1}(\partial M)$
nondegenerate (hence stable) critical points, where $P_{1}(\partial M)$
is the Poincar\'e polynomial of $\partial M$, namely $P_{t}(\partial M)$,
evaluated in $t=1$. Hence, generically with respect to metric, Problem
(\ref{eq:Pnn}) has $P_{1}(\partial M)$ solution and holds $P_{1}(\partial M)\ge\text{cat}\partial M$.
Also, in many cases the strict inequality $P_{1}(\partial M)>\text{cat}\partial M$
holds.

The paper is organized as follows. In Section \ref{sec:prel} we summarize
some result that are necessary to frame the problem. Namely, we recall
some well known notion of Remannian geometry, we introduce the variational
setting and we study some properties of the second equation of the
systems. In Section \ref{sec:reduction} we perform the finite dimensional
reduction and we sketch the prove of Theorem \ref{thm:main}. A collection
of technical results is contained in Appendix \ref{app}.

\section{\label{sec:prel}Preliminary results}

We recall some well known fact about Riemannian manifold with boundary.

First of all we define the Fermi coordinate chart. 
\begin{defn}
If $q$ belongs to the boundary $\partial M$, let $\bar{y}=\left(z_{1},\dots,z_{n-1}\right)$
be Riemannian normal coordinates on the $n-1$ manifold $\partial M$
at the point $q$. For a point $\xi\in M$ close to $q$, there exists
a unique $\bar{\xi}\in\partial M$ such that $d_{g}(\xi,\partial M)=d_{g}(\xi,\bar{\xi})$.
We set $\bar{z}(\xi)\in\mathbb{R}^{n-1}$ the normal coordinates for
$\bar{\xi}$ and $z_{n}(\xi)=d_{g}(\xi,\partial M)$. Then we define
a chart $\Psi_{q}^{\partial}:\mathbb{R}_{+}^{n}\rightarrow M$ such
that $\left(\bar{z}(\xi),z_{n}(\xi)\right)=\left(\Psi_{q}^{\partial}\right)^{-1}(\xi)$.
These coordinates are called \emph{Fermi coordinates} at $q\in\partial M$.
The Riemannian metric $g_{q}\left(\bar{z},z_{n}\right)$ read through
the Fermi coordinates satisfies $g_{q}(0)=\Id$. 
\end{defn}
We note by $d_{g}^{\partial}$ and $\exp^{\partial}$ respectively
the geodesic distance and the exponential map on by $\partial M$.
By compactness of $\partial M$, there is an $R^{\partial}$ and a
finite number of points $q_{i}\in\partial M$, $i=1,\dots,k$ such
that 
\[
I_{q_{i}}(R^{\partial},R_{M}):=\left\{ x\in M,\, d_{g}(x,\partial M)=d_{g}(x,\bar{\xi})<R_{M},\, d_{g}^{\partial}(q_{i},\bar{\xi})<R^{\partial}\right\} 
\]
 form a covering of $\left(\partial M\right)_{\rho}$ and on every
$I_{q_{i}}$ the Fermi coordinates are well defined. In the following
we choose, $R=\min\left\{ R^{\partial},R_{M}\right\} $, such that
we have a finite covering 
\[
M\subset\left\{ \cup_{i=1}^{k}B(q_{i},R)\right\} \bigcup\left\{ \cup_{i=k+1}^{l}I_{\xi_{i}}(R,R)\right\} 
\]
where $k,l\in\mathbb{N}$, $q_{i}\in M\smallsetminus\partial M$ and
$\xi_{i}\in\partial M$.

Given the Fermi coordinates in a neighborhood of $p$, and we denoted
by the matrix $(h_{ij})_{i,j=1,\dots,n-1}$ the second fundamental
form, we have the well known formulas (see \cite{BP05,Es92})
\begin{eqnarray}
g^{ij}(y) & = & \delta_{ij}+2h_{ij}(0)y_{n}+O(|y|^{2})\text{ for }i,j=1,\dots n-1\label{eq:g1}\\
g^{in}(y) & = & \delta_{in}\label{eq:g2}\\
\sqrt{g}(y) & = & 1-(n-1)H(0)y_{n}+O(|y|^{2})\label{eq:g3}
\end{eqnarray}
where $(y_{1},\dots,y_{n})$ are the Fermi coordinates and the mean
curvature $H$ is 
\begin{equation}
H=\frac{1}{n-1}\sum_{i}^{n-1}h_{ii}\label{eq:H}
\end{equation}

To solve our system, using an idea of Benci and Fortunato \cite{BF},
we reduce the system to a single equation. We introduce the map $\psi$
defined by the equation
\begin{equation}
\left\{ \begin{array}{cc}
-\Delta_{g}\psi+(1+q^{2}u^{2})\psi=qu^{2} & \text{ in }M\\
\frac{\partial\psi}{\partial\nu}=0 & \text{on }\partial M
\end{array}\right.\label{eq:ei-N}
\end{equation}
in case of Neumann boundary condition or by
\begin{equation}
\left\{ \begin{array}{cc}
-\Delta_{g}\psi+qu^{2}\psi=qu^{2} & \text{ in }M\\
\psi=0 & \text{on }\partial M
\end{array}\right.\label{eq:ei-D}
\end{equation}
in case of Dirichlet boundary condition. 

In what follows we call $H=H_{g}^{1}$ for the Neumann problem and
$H=H_{0,g}^{1}$ for the Dirichlet problem. Thus with abuse of language
we will say that $\psi:H\rightarrow H$ in both (\ref{eq:ei-N}) and
(\ref{eq:ei-D}). Moreover, from standard variational arguments, it
easy to see that $\psi$ is well-defined in $H$ and it holds 
\begin{equation}
0\le\psi(u)\le1/q\label{psipos}
\end{equation}
for all $u\in H$. We collect now some well known result on the map
$\psi$. For a more extensive presentation of these properties we
refer to \cite{DH}
\begin{lem}
\label{lem:e1}The map $\psi:H\rightarrow H$ is $C^{2}$ and its
differential $\psi'(u)[h]=V_{u}[h]$ at $u$ is the map defined by
\begin{equation}
-\Delta_{g}V_{u}[h]+(1+q^{2}u^{2})V_{u}[h]=2qu(1-q\psi(u))h\text{ for all }h\in H.\label{eq:e2}
\end{equation}
in case of Neumann boundary condition or 
\begin{equation}
-\Delta_{g}V_{u}[h]+q^{2}u^{2}V_{u}[h]=2qu(1-q\psi(u))h\text{ for all }h\in H.\label{eq:e2-2}
\end{equation}
in case of Dirichlet boundary condition.

Also, we have 
\[
0\le\psi'(u)[u]\le\frac{2}{q}.
\]
Finally, the second derivative $(h,k)\rightarrow\psi''(u)[h,k]=T_{u}(h,k)$
is the map defined by the equation 
\[
-\Delta_{g}T_{u}(h,k)+(1+q^{2}u^{2})T_{u}(h,k)=-2q^{2}u(kV_{u}(h)+hV_{u}(k))+2q(1-q\psi(u))hk
\]
in case of Neumann boundary condition or
\[
-\Delta_{g}T_{u}(h,k)+q^{2}u^{2}T_{u}(h,k)=-2q^{2}u(kV_{u}(h)+hV_{u}(k))+2q(1-q\psi(u))hk
\]
in case of Dirichlet boundary condition.
\end{lem}

\begin{lem}
\label{lem:e2}The map $\Theta:H\rightarrow\mathbb{R}$ given by 
\[
\Theta(u)=\frac{1}{2}\int_{M}(1-q\psi(u))u^{2}d\mu_{g}
\]
is $C^{2}$ and 
\[
\Theta'(u)[h]=\int_{M}(1-q\psi(u))^{2}uhd\mu_{g}
\]
for any $u,h\in H$ 
\end{lem}
For the proofs of these results we refer to , in which the case of
KGMP is treated. For KGM systems, the proof is identical.

Now, we introduce the functionals $I_{\varepsilon},J_{\varepsilon},G_{\varepsilon}:H\rightarrow\mathbb{R}$
\begin{equation}
I_{\varepsilon}(u)=J_{\varepsilon}(u)+\frac{\omega^{2}}{2}G_{\varepsilon}(u),\label{ieps-1}
\end{equation}
where 
\begin{equation}
J_{\varepsilon}(u):=\frac{1}{2\varepsilon^{n}}\int\limits _{M}\left[\varepsilon^{2}|\nabla_{g}u|^{2}+(a-\omega^{2})u^{2}\right]d\mu_{g}-\frac{1}{p\varepsilon^{n}}\int\limits _{M}\left(u^{+}\right)^{p}d\mu_{g}\label{jieps-1}
\end{equation}
and
\begin{equation}
G_{\varepsilon}(u):=\frac{1}{\varepsilon^{n}}q\int_{M}\psi(u)u^{2}d\mu_{g}.\label{geps-1}
\end{equation}
By Lemma \ref{lem:e2} we deduce that 
\begin{equation}
\frac{1}{2}G_{\varepsilon}'(u)[\varphi]=\frac{1}{\varepsilon^{n}}\int_{M}[2q\psi(u)-q^{2}\psi^{2}(u)]u\varphi d\mu_{g}.\label{eq:gprimo}
\end{equation}
If $u\in H$ is a critical point of $I_{\varepsilon}$ then the pair
$(u,\psi(u))$ is the desired solution of Problem (\ref{eq:Pnn})
or (\ref{eq:Pdn}).

Finally, we introduce a model function for the solution $u$. It is
well known that, in $\mathbb{R}^{n}$, there is a unique positive
radially symmetric function $V(z)\in H^{1}(\mathbb{R}^{n})$ satisfying
\begin{equation}
-\Delta V+(a-\omega^{2})V=V^{p-1}\text{ on }\mathbb{R}^{n}.\label{eq:rn}
\end{equation}
 Moreover, the function $V$ exponentially decays at infinity as well
as its derivative, that is, for some $c>0$ 
\begin{eqnarray*}
\lim_{|z|\rightarrow\infty}V(|z|)|z|^{\frac{n-1}{2}}e^{|z|}=c &  & \lim_{|z|\rightarrow\infty}V'(|z|)|z|^{\frac{n-1}{2}}e^{|z|}=-c.
\end{eqnarray*}
We can define on the half space $\mathbb{R}_{+}^{n}=\left\{ (z_{1,}\dots,z_{n})\in\mathbb{R}^{n}\ ,\ z_{n}\ge0\right\} $
the function
\[
U(x)=\left.V\right|_{x_{n}\ge0}.
\]
 The function $U$satisfies the following Neumann problem in $\mathbb{R}_{+}^{n}$
\begin{equation}
\left\{ \begin{array}{cc}
-\Delta U+(a-\omega^{2})U=U^{p-1} & \text{in }\mathbb{R}_{+}^{n}\\
{\displaystyle \frac{\partial U}{\partial z_{n}}=0} & \text{on }\left\{ z_{n}=0\right\} .
\end{array}\right.\label{eq:P-Rn}
\end{equation}
and it is easy to see that the space solution of the linearized problem
\begin{equation}
\left\{ \begin{array}{cc}
-\Delta\varphi+(a-\omega^{2})\varphi=(p-1)U^{p-2}\varphi & \text{ in }\mathbb{R}_{+}^{n}\\
{\displaystyle \frac{\partial\varphi}{\partial z_{n}}=0} & \text{on }\left\{ z_{n}=0\right\} .
\end{array}\right.\label{eq:linear}
\end{equation}
 is generated by the linear combination of 
\[
\varphi^{i}=\frac{\partial U}{\partial z_{i}}(z)\text{ for }i=1,\dots,n-1.
\]
We endow $H_{g}^{1}(M)$ with the scalar product ${\displaystyle \left\langle u,v\right\rangle _{\varepsilon}:=\frac{1}{\varepsilon^{n}}\int_{M}\varepsilon^{2}\nabla_{g}u\nabla_{g}v+(a-\omega^{2})uvd\mu_{g}}$
and the norm $\|u\|_{\varepsilon}=\left\langle u,u\right\rangle _{\varepsilon}^{1/2}$.
We call $H_{\varepsilon}$ the space $H_{g}^{1}$ equipped with the
norm $\|\cdot\|_{\varepsilon}$. We also define $L_{\varepsilon}^{p}$
as the space $L_{g}^{p}(M)$ endowed with the norm ${\displaystyle |u|_{\varepsilon,p}=\frac{1}{\varepsilon^{n}}\left(\int_{M}u^{p}d\mu_{g}\right)^{1/p}}$. 

For any $p\in[2,2^{*})$, the embedding $i_{\varepsilon}:H_{\varepsilon}\hookrightarrow L_{\varepsilon,p}$
is a compact, continuous map, and it holds $|u|_{\varepsilon,p}\le c\|u\|_{\varepsilon}$
for some constant $c$ not depending on $\varepsilon$. We define
the adjoint operator $i_{\varepsilon}^{*}:L_{\varepsilon,p'}:\hookrightarrow H_{\varepsilon}$
as 
\[
u=i_{\varepsilon}^{*}(v)\ \Leftrightarrow\ \left\langle u,\varphi\right\rangle _{\varepsilon}=\frac{1}{\varepsilon^{n}}\int_{M}v\varphi d\mu_{g}.
\]
Now on set
\[
f(u)=|u^{+}|^{p-1}
\]
and 
\[
g(u):=\left(q^{2}\psi^{2}(u)-2q\psi(u)\right)u.
\]
we can rewrite problem (\ref{eq:Pnn}) in an equivalent formulation
\[
u=i_{\varepsilon}^{*}\left[f(u)+\omega^{2}g(u)\right],\ u\in H_{\varepsilon}.
\]
 
\begin{rem}
\label{rem:ieps}We have that $\|i_{\varepsilon}^{*}(v)\|_{\varepsilon}\le c|v|_{p',\varepsilon}$
with $c$ independent by $\varepsilon$.
\end{rem}

\begin{rem}
\label{rem:normaeps}We recall the following two estimates, that can
be obtained by trivial computations
\begin{align}
\|u\|_{H_{g}^{1}} & \le c\varepsilon^{\frac{1}{2}}\|u\|_{\varepsilon}\text{ for }n=3\label{eq:norma1}\\
\|u\|_{H_{g}^{1}} & \le c\varepsilon\|u\|_{\varepsilon}\text{ for }n=4\label{eq:norma2}
\end{align}
We often will use the estimate (\ref{eq:norma1}) also when $n=4$,
which is still true even if weaker, to simplify the exposition.
\end{rem}
Finally, we define an important class of functions on the manifold,
modeled on the function $U$. For all $\xi\in\partial M$ we define
\[
W_{\varepsilon,\xi}=\left\{ \begin{array}{ccc}
U_{\varepsilon}\left(\left(\Psi_{\xi}^{\partial}\right)^{-1}(x)\right)\chi_{R}\left(\left(\Psi_{\xi}^{\partial}\right)^{-1}(x)\right) &  & x\in I_{\xi}(R):=I_{\xi}(R,R);\\
0 &  & \text{elsewhere}.
\end{array}\right.
\]
We recall a fundamental limit property for the function $W_{\varepsilon,\xi}$.
\begin{rem}
\label{remark:Weps}Since $U$ decays exponentially, it holds, uniformly
with respect to $q\in\partial M$, 
\begin{equation}
\lim_{\varepsilon\rightarrow0}\left|W_{\varepsilon,\xi}\right|_{t,\varepsilon}^{t}=\int_{\mathbb{R}_{+}^{n}}U^{t}(z)dz\label{eql2-1}
\end{equation}
for all $1\le t\le2^{*}$, and
\begin{equation}
\lim_{\varepsilon\rightarrow0}\varepsilon^{2}\left|\nabla_{g}W_{\varepsilon,\xi}\right|_{2,\varepsilon}^{2}=\int_{\mathbb{R}_{+}^{n}}\left\vert \nabla U\right\vert ^{2}(z)dz\label{eqgrad-1}
\end{equation}
We also have the following estimate for the function $\psi$ and for
its differential $\psi'$. \end{rem}
\begin{lem}
It holds, for any $\varphi\in H$ and for any $\xi\in\partial M$
\begin{equation}
\|\psi(W_{\varepsilon,\xi}+\varphi)\|_{H}\le c_{1}\left(\varepsilon^{\frac{n+2}{2}}+\|\varphi\|_{H}^{2}\right)\label{eq:phiweps-1}
\end{equation}
\begin{equation}
\|\psi(W_{\varepsilon,\xi}+\varphi)\|_{H}\le c_{2}\varepsilon^{\frac{n+2}{2}}\left(1+\|\varphi\|_{\varepsilon}^{2}\right)\label{eq:phiweps-2}
\end{equation}
for some positive constants $c_{1},c_{2}$, when $\varepsilon$ is
sufficiently small.\end{lem}
\begin{proof}
We prove the claim for the Neumann boundary condition. For the Dirichlet
boundary condition the proof is completely analogous taking in account
the gradient norm on $H.$

To simplify the notations we set $v=\psi(W_{\varepsilon,\xi}+\varphi)$.
By definition of $\psi$ we have 
\begin{align*}
\|v\|_{H}^{2} & \le  \int_{M}|\nabla_{g}v|^{2}+v^{2}+q^{2}(W_{\varepsilon,\xi}+\varphi)^{2}v^{2}
=q\int(W_{\varepsilon,\xi}+\varphi)^{2}v\\
 & \le  \left(\int_{M}v^{2^{*}}\right)^{\frac{1}{2^{*}}}
 \left(\int_{M}(W_{\varepsilon,\xi}+\varphi)^{\frac{4n}{n+2}}\right)^{\frac{n+2}{2n}}
 \le c\|v\|_{H_{g}^{1}}\left|W_{\varepsilon,\xi}+\varphi\right|_{\frac{4n}{n+2},g}^{2}\\
 & \le c\|v\|_{H_{g}^{1}}\left(\left|W_{\varepsilon,\xi}\right|_{\frac{4n}{n+2},g}^{2}+\left|\varphi\right|_{\frac{4n}{n+2},g}^{2}\right)
\end{align*}
Thus $\|v\|_{H}\le c\left(\left|W_{\varepsilon,\xi}\right|_{\frac{4n}{n+2},g}+\left|\varphi\right|_{\frac{4n}{n+2},g}\right)$.
Taking in account (\ref{eql2-1}) of Remark \ref{remark:Weps} we
have that, for $\varepsilon$ small $\left|W_{\varepsilon,\xi}\right|_{\frac{4n}{n+2},g}^{2}\le C\varepsilon^{\frac{2n}{n+2}}\left|U\right|_{\frac{4n}{n+2},g}^{2}$.
Thus we have 

\begin{equation}
\|v\|_{H_{g}^{1}}\le c_{1}\left(\varepsilon^{\frac{2n}{n+2}}+|\varphi|_{\frac{4n}{n+2},g}^{2}\right)\le c_{1}(\varepsilon^{\frac{2n}{n+2}}+\|\varphi\|_{H_{g}^{1}}^{2})\label{eq:e6bis}
\end{equation}
and 
\begin{equation}
\|v\|_{H_{g}^{1}}\le c_{2}\varepsilon^{\frac{2n}{n+2}}\left(1+|\varphi|_{\frac{4n}{n+2},\varepsilon}^{2}\right)\le c_{2}\varepsilon^{\frac{2n}{n+2}}\left(1+\|\varphi\|_{\varepsilon}^{2}\right).\label{eq:e6}
\end{equation}
that prove (\ref{eq:phiweps-1}) and (\ref{eq:phiweps-2}).\label{lem:e7-D}For
any $\xi\in M$ and $h,k\in H_{g}^{1}$ it holds\end{proof}
\begin{lem}
\label{lem:psiprimo}It holds, for any $h,k\in H$ and for any $\xi\in\partial M$

\[
\|\psi'(W_{\varepsilon,\xi}+k)[h]\|_{H}\le c\|h\|_{H}\left\{ \varepsilon^{2}+\|k\|_{H}\right\} 
\]
for some positive constant $c$when $\varepsilon$ is sufficiently
small.\end{lem}
\begin{proof}
Again, we prove the claim for the Neumann boundary condition being
the other case completely analogous. By (\ref{eq:e2}) and since $0<\psi<1/q$,
\begin{eqnarray*}
\|\psi'(W_{\varepsilon,\xi}+k)[h]\|_{H_{g}^{1}}^{2} & = & 2q\int_{M}(W_{\varepsilon,\xi}+k)(1-q\psi(W_{\varepsilon,\xi}+k))h\psi'(W_{\varepsilon,\xi}+k)[h]\\
 &  & -q^{2}\int_{M}(W_{\varepsilon,\xi}+k)^{2}(\psi'(W_{\varepsilon,\xi}+k)[h])^{2}\\
 & \le & \int_{M}W_{\varepsilon,\xi}|h|\left|\psi'(W_{\varepsilon,\xi}+k)[h]\right|+\int_{M}|k||h|\left|\psi'(W_{\varepsilon,\xi}+k)[h]\right|\\
 & := & I_{1}+I_{2}
\end{eqnarray*}
We estimate the two terms $I_{1}$ and $I_{2}$ separately. We have
\[
I_{1}\le\left|\psi'(W_{\varepsilon,\xi}+k)[h]\right|_{2^{*},g}\left|h\right|_{2^{*},g}\left|W_{\varepsilon,\xi}\right|_{\frac{2}{n},g}\le\varepsilon^{2}\|\psi'\|_{H_{g}^{1}}\|h\|_{H_{g}^{1}}\left|W_{\varepsilon,\xi}\right|_{\frac{n}{2},\varepsilon}
\]
\[
I_{2}\le\|k\|_{L_{g}^{3}}\|h\|_{L_{g}^{3}}\|\psi'(W_{\varepsilon,\xi}+k)[h]\|_{L_{g}^{3}}\le\|k\|_{H_{g}^{1}}\|h\|_{H_{g}^{1}}\|\psi'\|_{H_{g}^{1}}
\]
and, in light of Remark \ref{remark:Weps}, we obtain the claim.
\end{proof}

\subsection{The Lyapunov Schmidt reduction}

We want to split the space $H_{\varepsilon}$ in a finite dimensional
space generated by the solution of (\ref{eq:linear}) and its orthogonal
complement. Fixed $\xi\in\partial M$ and $R>0$, we consider on the
manifold the functions 
\begin{equation}
Z_{\varepsilon,\xi}^{i}=\left\{ \begin{array}{ccc}
\varphi_{\varepsilon}^{i}\left(\left(\psi_{\xi}^{\partial}\right)^{-1}(x)\right)\chi_{R}\left(\left(\psi_{\xi}^{\partial}\right)^{-1}(x)\right) &  & x\in I_{\xi}(R):=I_{\xi}(R,R);\\
0 &  & \text{elsewhere}.
\end{array}\right.\label{eq:Zi}
\end{equation}
where ${\displaystyle \varphi_{\varepsilon}^{i}(z)=\varphi^{i}\left(\frac{z}{\varepsilon}\right)}$
and $\chi_{R}:B^{n-1}(0,R)\times[0,R)\rightarrow\mathbb{R}^{+}$ is
a smooth cut off function such that $\chi_{R}\equiv1$ on $B^{n-1}(0,R/2)\times[0,R/2)$
and $|\nabla\chi|\le2$.

In the following, for sake of simplicity, we denote
\begin{equation}
D^{+}(R)=B^{n-1}(0,R)\times[0,R)\subset\mathbb{R}_{+}^{n}\label{eq:D+}
\end{equation}

Let 
\[
K_{\varepsilon,\xi}:=\mbox{Span}\left\{ Z_{\varepsilon,\xi}^{1},\cdots,Z_{\varepsilon,\xi}^{n-1}\right\} .
\]
 We can split $H_{\varepsilon}$ in the sum of the $\left(n-1\right)$-dimensional
space and its orthogonal complement with respect of $\left\langle \cdot,\cdot\right\rangle _{\varepsilon}$,
i.e.
\[
K_{\varepsilon,\xi}^{\bot}:=\left\{ u\in H_{\varepsilon}\ ,\ \left\langle u,Z_{\varepsilon,\xi}^{i}\right\rangle _{\varepsilon}=0.\right\} .
\]
We solve problem (\ref{eq:Pnn}) by a Lyapunov Schmidt reduction:
we look for a function of the form $W_{\varepsilon,\xi}+\phi$ with
$\phi\in K_{\varepsilon,\xi}^{\bot}$ such that 
\begin{eqnarray}
\Pi_{\varepsilon,\xi}^{\bot}\left\{ W_{\varepsilon,\xi}+\phi-i_{\varepsilon}^{*}\left[f\left(W_{\varepsilon,\xi}+\phi\right)+\omega^{2}g\left(W_{\varepsilon,\xi}+\phi\right)\right]\right\}  & = & 0\label{eq:red1}\\
\Pi_{\varepsilon,\xi}\left\{ W_{\varepsilon,\xi}+\phi-i_{\varepsilon}^{*}\left[f\left(W_{\varepsilon,\xi}+\phi\right)+\omega^{2}g\left(W_{\varepsilon,\xi}+\phi\right)\right]\right\}  & = & 0\label{eq:red2}
\end{eqnarray}
where $\Pi_{\varepsilon,\xi}:H_{\varepsilon}\rightarrow K_{\varepsilon,\xi}$
and $\Pi_{\varepsilon,\xi}^{\bot}:H_{\varepsilon}\rightarrow K_{\varepsilon,\xi}^{\bot}$
are, respectively, the projection on $K_{\varepsilon,\xi}$ and $K_{\varepsilon,\xi}^{\bot}$.
We see that $W_{\varepsilon,\xi}+\phi$ is a solution of (\ref{eq:Pnn})
if and only if $W_{\varepsilon,\xi}+\phi$ solves (\ref{eq:red1}-\ref{eq:red2}).

\section{\label{sec:reduction}Reduction to finite dimensional space}

In this section we find a solution for equation (\ref{eq:red1}).
In particular, we prove that for all $\varepsilon>0$ and for all
$\xi\in\partial M$ there exists $\phi_{\varepsilon,\xi}\in K_{\varepsilon,\xi}^{\bot}$
solving (\ref{eq:red1}). The main part of the reduction is performed
in \cite{GMbordo} and in \cite{MP09}. Here we explicitly estimate
only the term appearing in this specific contest.

We can rewrite equation (\ref{eq:red1}) as 

\[
L_{\varepsilon,\xi}(\phi)=N_{\varepsilon,\xi}(\phi)+R_{\varepsilon,\xi}+S_{\varepsilon,\xi}(\phi)
\]
were $L_{\varepsilon,\xi}$ is the linear operator
\begin{eqnarray*}
L_{\varepsilon,\xi} & : & K_{\varepsilon,\xi}^{\bot}\rightarrow K_{\varepsilon,\xi}^{\bot}\\
L_{\varepsilon,\xi}(\phi) & := & \Pi_{\varepsilon,\xi}^{\bot}\left\{ \phi-i_{\varepsilon}^{*}\left[f'(W_{\varepsilon,\xi})\phi\right]\right\} ,
\end{eqnarray*}
$N_{\varepsilon,\xi}(\phi)$ is the nonlinear term 
\[
N_{\varepsilon,\xi}:=\Pi_{\varepsilon,\xi}^{\bot}\left\{ i_{\varepsilon}^{*}\left[f(W_{\varepsilon,\xi}+\phi)-f(W_{\varepsilon,\xi})-f'(W_{\varepsilon,\xi})\phi\right]\right\} 
\]
$R_{\varepsilon,\xi}$ is a remainder term
\[
R_{\varepsilon,\xi}:=\Pi_{\varepsilon,\xi}^{\bot}\left\{ i_{\varepsilon}^{*}\left[f(W_{\varepsilon,\xi})\right]-W_{\varepsilon,\xi}\right\} 
\]
and $S_{\varepsilon,\xi}$ is the coupling term 
\[
S_{\varepsilon,\xi}=\Pi_{\varepsilon,\xi}^{\bot}\left\{ i_{\varepsilon}^{*}\left[\omega^{2}g\left(W_{\varepsilon,\xi}+\phi\right)\right]\right\} .
\]

\begin{prop}premise
\label{prop:phieps}There exists $\varepsilon_{0}>0$ and $C>0$ such
that for any $\xi\in\partial M$ and for all $\varepsilon\in(0,\varepsilon_{0})$
there exists a unique $\phi_{\varepsilon,\xi}=\phi(\varepsilon,\xi)\in K_{\varepsilon,\xi}^{\bot}$
which solves (\ref{eq:red1}). Moreover
\[
\|\phi_{\varepsilon,\xi}\|_{\varepsilon}<C\varepsilon^{2}.
\]
Finally, $\xi\mapsto\phi_{\varepsilon,\xi}$ is a $C^{1}$ map.
\end{prop}
To prove this result, we premise some technical lemma.
\begin{rem}
\label{lem:Linv}We summarize here the results on $L_{\varepsilon,\xi},N_{\varepsilon,\xi}$
and $R_{\varepsilon,\xi}$ contained in \cite{GMbordo}.

There exist $\varepsilon_{0}$ and $c>0$ such that, for any $\xi\in\partial M$
and $\varepsilon\in(0,\varepsilon_{0})$ 
\[
\|L_{\varepsilon,\xi}\|_{\varepsilon}\geq c\|\phi\|_{\varepsilon}\text{ for any }\phi\in K_{\varepsilon,\xi}^{\bot}.
\]
Also it holds 
\[
\|R_{\varepsilon,\xi}\|_{\varepsilon}\le c\varepsilon^{1+\frac{n}{p'}}
\]
and 
\[
\|N_{\varepsilon,\xi}(\phi)\|_{\varepsilon}\le c\left(\|\phi\|_{\varepsilon}^{2}+\|\phi\|_{\varepsilon}^{p-1}\right)
\]
We further remark that $\frac{n}{p'}>1$ since $2\le p<2^{*}$ 
\end{rem}
We have now to estimate the coupling term $S_{\varepsilon,\xi}$. 
\begin{lem}
If $\|\phi\|_{\varepsilon},\|\phi_{1}\|_{\varepsilon},\|\phi_{2}\|_{\varepsilon}=O(\varepsilon^{2})$
it holds 
\begin{eqnarray}
\|S_{\varepsilon,\xi}(\phi)\|_{\varepsilon} & \le & c\varepsilon^{2}\label{eq:Seps1}\\
\|S_{\varepsilon,\xi}(\phi_{1})-S_{\varepsilon,\xi}(\phi_{2})\|_{\varepsilon} & \le & l_{\varepsilon}\|\phi_{1}-\phi_{2}\|_{\varepsilon}\label{eq:Seps2}
\end{eqnarray}
where $l_{\varepsilon}\rightarrow0$ as $\varepsilon\rightarrow0$. \end{lem}
\begin{proof}
We have, by the properties of the map $i_{\varepsilon}^{*}$, that
\begin{eqnarray*}
\|S_{\varepsilon,\xi}(\phi)\|_{\varepsilon} & \le & c\left|\psi^{2}(W_{\varepsilon,\xi}+\phi)(W_{\varepsilon,\xi}+\phi)\right|_{\varepsilon,p'}+c\left|\psi(W_{\varepsilon,\xi}+\phi)(W_{\varepsilon,\xi}+\phi)\right|_{\varepsilon,p'}\\
 & \le & c\left|\psi(W_{\varepsilon,\xi}+\phi)(W_{\varepsilon,\xi}+\phi)\right|_{\varepsilon,p'}\\
 & \le & \frac{c}{\varepsilon^{\frac{n}{p'}}}\left(\int\psi(W_{\varepsilon,\xi}+\phi)^{2*}\right)^{\frac{1}{2^{*}}}\left(\int|W_{\varepsilon,\xi}+\phi|^{p'\left(\frac{2^{*}}{p'}\right)^{'}}\right)^{\frac{1}{p^{'}\left(\frac{2^{*}}{p'}\right)^{'}}}\\
 & \le c & \varepsilon^{-\frac{n}{p'}+\frac{n}{p^{'}\left(\frac{2^{*}}{p'}\right)^{'}}}\|\psi(W_{\varepsilon,\xi}+\phi)\|_{H}\left|W_{\varepsilon,\xi}+\phi\right|_{\varepsilon,p'\left(\frac{2^{*}}{p'}\right)^{'}}\\
 &\le& c\varepsilon^{-\frac{n}{2^{*}}}\|\psi(W_{\varepsilon,\xi}+\phi)\|_{H}
 \le c\varepsilon^{-\frac{n}{2^{*}}}\varepsilon^{\frac{n+2}{2}}=c\varepsilon^{2}
\end{eqnarray*}
by (\ref{eq:phiweps-2}) and taking in account that $\|\phi\|_{\varepsilon}=o(1)$
by Remark \ref{rem:normaeps}, and the first step is proved. 

For the second claim, we have, since $0\le\psi\le1/q$
\begin{align*}
\|S_{\varepsilon,\xi}(\phi_{1})-S_{\varepsilon,\xi}(\phi_{2})\|_{\varepsilon}\le & c\left|\psi^{2}(W_{\varepsilon,\xi}+\phi_{1})(W_{\varepsilon,\xi}+\phi_{1})-\psi^{2}(W_{\varepsilon,\xi}+\phi_{2})(W_{\varepsilon,\xi}+\phi_{2})\right|_{\varepsilon,p'}\\
 & +c\left|\psi(W_{\varepsilon,\xi}+\phi_{1})(W_{\varepsilon,\xi}+\phi_{1})-\psi(W_{\varepsilon,\xi}+\phi_{2})(W_{\varepsilon,\xi}+\phi_{2})\right|_{\varepsilon,p'}\\
\le & c\left|\psi(W_{\varepsilon,\xi}+\phi_{1})(W_{\varepsilon,\xi}+\phi_{1})-\psi(W_{\varepsilon,\xi}+\phi_{2})(W_{\varepsilon,\xi}+\phi_{2})\right|_{\varepsilon,p'}\\
\le & c\left|\left[\psi(W_{\varepsilon,\xi}+\phi_{1})-\psi(W_{\varepsilon,\xi}+\phi_{2})\right](W_{\varepsilon,\xi}+\phi_{1})\right|_{\varepsilon,p'}\\
 & +\left|\psi(W_{\varepsilon,\xi}+\phi_{2})\left[\phi_{1}-\phi_{2}\right]\right|_{\varepsilon,p'}\\
\le & c\left|\left[\psi'(W_{\varepsilon,\xi}+(1-\theta)\phi_{1}+\theta\phi_{2})[\phi_{1}-\phi_{2}]\right](W_{\varepsilon,\xi}+\phi_{1})\right|_{\varepsilon,p'}\\
 & +\left|\psi(W_{\varepsilon,\xi}+\phi_{2})\left[\phi_{1}-\phi_{2}\right]\right|_{\varepsilon,p'}:=D_{1}+D_{2}
\end{align*}
for some $\theta\in(0,1)$. Arguing as in the first part of the proof
we get, in light of (\ref{eq:phiweps-2}), that 
\[
D_{2}\le c\varepsilon^{-\frac{n}{p^{*}}}\|\psi(W_{\varepsilon,\xi}+\phi)\|_{H}\left|\phi_{1}-\phi_{2}\right|_{\varepsilon,p'\left(\frac{2^{*}}{p'}\right)^{'}}\le c\varepsilon^{-\frac{n}{2^{*}}}\varepsilon^{\frac{n+2}{2}}\|\phi_{1}-\phi_{2}\|_{\varepsilon}
\]
 and, using Lemma \ref{lem:psiprimo}, that 
\begin{align*}
D_{1} & \le c\varepsilon^{-\frac{n}{p'}}\left\Vert \left[\psi'(W_{\varepsilon,\xi}+(1-\theta)\phi_{1}+\theta\phi_{2})[\phi_{1}-\phi_{2}]\right]\right\Vert _{H}\left|W_{\varepsilon,\xi}+\phi_{1}\right|_{\varepsilon,p'\left(\frac{2^{*}}{p'}\right)^{'}}\\
 & \le c\varepsilon^{-\frac{n}{p'}}\left\{ \varepsilon^{2}+(1-\theta)\|\phi_{1}\|_{H}+\theta\|\phi_{2}\|_{H}\right\} \|\phi_{1}-\phi_{2}\|_{H}.
\end{align*}
If $n=3,$ by (\ref{eq:norma1}) and since $\|\phi_{1}\|_{\varepsilon},\|\phi_{2}\|_{\varepsilon}=o(\varepsilon)$
by hypothesis we have
\begin{align*}
D_{1} & \le c\varepsilon^{-\frac{3}{p'}}\left\{ \varepsilon^{2}+\varepsilon^{1/2}(1-\theta)\|\phi_{1}\|_{\varepsilon}+\varepsilon^{1/2}\theta\|\phi_{2}\|_{\varepsilon}\right\} \varepsilon^{1/2}\|\phi_{1}-\phi_{2}\|_{\varepsilon}\\
 & \le c\varepsilon^{\frac{5}{2}-\frac{3}{p'}}\|\phi_{1}-\phi_{2}\|_{\varepsilon}
\end{align*}
and the claim is proved since $\frac{5}{2}-\frac{3}{p'}>0$ if $p'>\frac{6}{5}$
that is true since $p<6$. For $n=4,$ analogously we have, by (\ref{eq:norma2})
\begin{align*}
D_{1} & \le c\varepsilon^{-\frac{4}{p'}}\left\{ \varepsilon^{2}+\varepsilon(1-\theta)\|\phi_{1}\|_{\varepsilon}+\varepsilon\theta\|\phi_{2}\|_{\varepsilon}\right\} \varepsilon\|\phi_{1}-\phi_{2}\|_{\varepsilon}\\
 & \le c\varepsilon^{3-\frac{4}{p'}}\|\phi_{1}-\phi_{2}\|_{\varepsilon}
\end{align*}
and $3-\frac{4}{p'}>0$ iff $p'>\frac{4}{3}$ that is $p<4$.
\end{proof}

We can now prove the main result of this section
\begin{proof}[Proof of Proposition \ref{prop:phieps}]
The proof is similar to Proposition 3.5 of \cite{MP09}, which we
refer to for all details. We want to solve (\ref{eq:red1}) by a fixed
point argument. We define the operator
\begin{eqnarray*}
T_{\varepsilon,\xi} & : & K_{\varepsilon,\xi}^{\bot}\rightarrow K_{\varepsilon,\xi}^{\bot}\\
T_{\varepsilon,\xi}(\phi) & = & L_{\varepsilon,\xi}^{-1}\left(N_{\varepsilon,\xi}(\phi)+R_{\varepsilon,\xi}S_{\varepsilon,\xi}(\phi)\right)
\end{eqnarray*}
By Remark \ref{lem:Linv} $T_{\varepsilon,\xi}$ is well defined and
it holds
\begin{eqnarray*}
\|T_{\varepsilon,\xi}(\phi)\|_{\varepsilon} & \le & c\left(\|N_{\varepsilon,\xi}(\phi)\|_{\varepsilon}+\|R_{\varepsilon,\xi}\|_{\varepsilon}+\|S_{\varepsilon,\xi}(\phi)\|_{\varepsilon}\right)\\
\|T_{\varepsilon,\xi}(\phi_{1})-T_{\varepsilon,\xi}(\phi_{2})\|_{\varepsilon} & \le & c\left(\|N_{\varepsilon,\xi}(\phi_{1})-N_{\varepsilon,\xi}(\phi_{2})\|_{\varepsilon}+\|S_{\varepsilon,\xi}(\phi_{1})-S_{\varepsilon,\xi}(\phi_{2})\|_{\varepsilon}\right)
\end{eqnarray*}
for some suitable constant $c>0$. By the mean value theorem (and
by the properties of $i^{*}$) we get
\[
\|N_{\varepsilon,\xi}(\phi_{1})-N_{\varepsilon,\xi}(\phi_{2})\|_{\varepsilon}\le c\left|f'(W_{\varepsilon,\xi}+\phi_{2}+t(\phi_{1}-\phi_{2}))-f'(W_{\varepsilon,\xi})\right|_{\frac{p}{p-2},\varepsilon}\|\phi_{1}-\phi_{2}\|_{\varepsilon}.
\]
By \cite{MP09}, Remark 3.4 we have that $\left|f'(W_{\varepsilon,\xi}+\phi_{2}+t(\phi_{1}-\phi_{2}))-f'(W_{\varepsilon,\xi})\right|_{\frac{p}{p-2},\varepsilon}<<1$
provided $\|\phi_{1}\|_{\varepsilon}$ and $\|\phi_{2}\|_{\varepsilon}$
small enough. This, combined with (\ref{eq:Seps2}) proves that there
exists $0<L<1$ such that $\|T_{\varepsilon,\xi}(\phi_{1})-T_{\varepsilon,\xi}(\phi_{2})\|_{\varepsilon}\le L\|\phi_{1}-\phi_{2}\|_{\varepsilon}$. 

We recall that by Lemma \ref{lem:Linv} we have 
\begin{eqnarray*}
\|N_{\varepsilon,\xi}(\phi)\|_{\varepsilon} & \le & c\left(\|\phi\|_{\varepsilon}^{2}+\|\phi\|_{\varepsilon}^{p-1}\right)\\
\|R_{\varepsilon,\xi}\|_{\varepsilon} & \le & \varepsilon^{1+\frac{n}{p'}}=o(\varepsilon^{2})
\end{eqnarray*}
This, combined with (\ref{eq:Seps1}) gives us 
\begin{align*}
\|T_{\varepsilon,\xi}(\phi)\|_{\varepsilon}&\le c\left(\|N_{\varepsilon,\xi}(\phi)\|_{\varepsilon}+\|R_{\varepsilon,\xi}\|_{\varepsilon}+\|S_{\varepsilon,\xi}(\phi)\|_{\varepsilon}\right)\\
&\le c\left(\|\phi\|_{\varepsilon}^{2}+\|\phi\|_{\varepsilon}^{p-1}+\varepsilon^{1+\frac{n}{p'}}+c\varepsilon^{2}\right)
\end{align*}
 So, there exists a positive constant $C$ such that $T_{\varepsilon,\xi}$
maps a ball of center $0$ and radius $C\varepsilon^{2}$ in $K_{\varepsilon,\xi}^{\bot}$
into itself and it is a contraction. So there exists a fixed point
$\phi_{\varepsilon,\xi}$ with norm $\|\phi_{\varepsilon,\xi}\|_{\varepsilon}\le C\varepsilon^{2}$. 

The continuity of $ $$\phi_{\varepsilon,\xi}$ with respect to $\xi$
is standard.
\end{proof}

\section{\label{sec:functional}The reduced functional}

In this section we define the reduced functional in a finite dimensional
space and we solve equation (\ref{eq:red2}). This leads us to the
prove of main theorem.

We have introduced $I_{\varepsilon}(u)$ in the introduction. We now
define the reduced functional 
\begin{eqnarray*}
\tilde{I}_{\varepsilon} & : & \partial M\rightarrow\mathbb{R}\\
\tilde{I}_{\varepsilon}(\xi) & = & I_{\varepsilon}(W_{\varepsilon,\xi}+\phi_{\varepsilon,\xi})
\end{eqnarray*}
where $\phi_{\varepsilon,\xi}$ is uniquely determined by Proposition
\ref{prop:phieps}.
\begin{lem}
\label{lem:sol2}Let $\xi_{0}$ a critical point of $\tilde{I}_{\varepsilon}$,
that is, if $\xi=\xi(y)=\exp_{\xi_{0}}^{\partial}(y)$, $y\in B^{n-1}(0,r)$,
then
\[
\left(\frac{\partial}{\partial y_{h}}\tilde{I}_{\varepsilon}(\xi(y))\right)_{|_{y=0}}=0,\ \ h=1,\dots,n-1.
\]

Thus the function $\phi_{\varepsilon,\xi}+W_{\varepsilon,\xi}$ solves
equation (\ref{eq:red2}). \end{lem}
\begin{proof}
The proof of this lemma is just a computation.
\end{proof}

\begin{lem}
\label{lem:Iexp}It holds
\[
\tilde{I}_{\varepsilon}(\xi)=C-\varepsilon H(\xi)+o(\varepsilon)
\]
$C^{1}$ uniformly with respect to $\xi\in\partial M$ as $\varepsilon$
goes to zero. Here $H(\xi)$ is the mean curvature of the boundary
$\partial M$ at $\xi$.
\end{lem}
To prove Lemma we study the asymptotic expansion of $\tilde{I}_{\varepsilon}(\xi)$
with respect to $\varepsilon$. We recall the result contained in
\cite{GMbordo}.
\begin{rem}
\label{lem:Jeps}It holds 
\begin{align}
\tilde{J}_{\varepsilon}(\xi): & =J_{\varepsilon}(W_{\varepsilon,\xi}+\phi_{\varepsilon,\xi})=J_{\varepsilon}(W_{\varepsilon,\xi})+o(\varepsilon)\label{eq:asexp1}\\
 & =C-\varepsilon\alpha H(\xi)+o(\varepsilon)\nonumber 
\end{align}
 $C^{1}$ uniformly with respect to $\xi\in\partial M$ as $\varepsilon$
goes to zero, where 
\begin{eqnarray*}
C & := & \int_{\mathbb{R}_{+}^{n}}\frac{1}{2}|\nabla U(z)|^{2}+\frac{1}{2}U^{2}(z)-\frac{1}{p}U^{p}(z)dz\\
\alpha & := & \frac{\left(n-1\right)}{2}\int_{\mathbb{R}_{+}^{n}}\left(\frac{U'(|z|)}{|z|}\right)^{2}z_{n}^{3}dz
\end{eqnarray*}

\end{rem}
In light of this result, it remains to estimate the coupling functional
$G_{\varepsilon}$ to prove Lemma \ref{lem:Iexp}. We split this proof
in several lemmas.
\begin{lem}
\label{lem:Geps1}It holds
\begin{equation}
G_{\varepsilon}\left(W_{\varepsilon,\xi}+\phi_{\varepsilon,\xi}\right)-G_{\varepsilon}\left(W_{\varepsilon,\xi}\right)=o(\varepsilon)\label{new1}
\end{equation}
\begin{equation}
\left[G'_{\varepsilon}\left(W_{\varepsilon,\xi_{0}}+\phi_{\varepsilon,\xi_{0}}\right)-G'_{\varepsilon}\left(W_{\varepsilon,\xi_{0}}\right)\right]\left[\left(\frac{\partial}{\partial y_{h}}W_{\varepsilon,\xi(y)}\right)_{|_{y=0}}\right]=o(\varepsilon)\label{new2}
\end{equation} 
\begin{equation}
G'_{\varepsilon}\left(W_{\varepsilon,\xi(y)}+\phi_{\varepsilon,\xi(y)}\right)\left[\frac{\partial}{\partial y_{h}}\phi_{\varepsilon,\xi(y)}\right]=o(\varepsilon)\label{new3}
\end{equation}
\end{lem}
\begin{proof}
Let us prove (\ref{new1}). We have (for some $\theta\in[0,1]$) 
\begin{multline*}
G_{\varepsilon}\left(W_{\varepsilon,\xi}+\phi_{\varepsilon,\xi}\right)-G_{\varepsilon}\left(W_{\varepsilon,\xi}\right)\\
=  \frac{1}{\varepsilon^{n}}\int\limits _{M}\left[\psi\left(W_{\varepsilon,\xi}+\phi_{\varepsilon,\xi}\right)\left(W_{\varepsilon,\xi}+\phi_{\varepsilon,\xi}\right)^{2}-\psi\left(W_{\varepsilon,\xi}\right)\left(W_{\varepsilon,\xi}\right)^{2}\right]\\
=  \frac{1}{\varepsilon^{n}}\int\limits _{M}\psi'\left(W_{\varepsilon,\xi}+\theta\phi_{\varepsilon,\xi}\right)[\phi_{\varepsilon,\xi}]\left(W_{\varepsilon,\xi}\right)^{2}\\
  +\frac{1}{\varepsilon^{n}}\int\limits _{M}\psi\left(W_{\varepsilon,\xi}+\phi_{\varepsilon,\xi}\right)\left(2\phi_{\varepsilon,\xi}W_{\varepsilon,\xi}+\phi_{\varepsilon,\xi}^{2}\right):=  I_{1}+I_{2}.
\end{multline*}
By Lemma \ref{lem:psiprimo} and Remark \ref{remark:Weps} we have
\begin{align*}
I_{1}\le & \frac{1}{\varepsilon^{n}}\left(\ensuremath{\int\limits _{M}}\left(\psi'\left(W_{\varepsilon,\xi}+\theta\phi_{\varepsilon,\xi}\right)[\phi_{\varepsilon,\xi}]\right)^{2}d\mu_{g}\right)^{\frac{1}{2}}\left(\ensuremath{\int\limits _{M}}W_{\varepsilon,\xi}^{4}d\mu_{g}\right)^{\frac{1}{2}}\\
\le & \frac{\varepsilon^{\frac{n}{2}}}{\varepsilon^{n}}\|\psi'\left(W_{\varepsilon,\xi}+\theta\phi_{\varepsilon,\xi}\right)[\phi_{\varepsilon,\xi}]\|_{H}\left|W_{\varepsilon,\xi}\right|_{\varepsilon,2}^{2}\\
\le & \varepsilon^{-\frac{n}{2}}\left(\varepsilon^{2}\|\phi_{\varepsilon,\xi}\|_{H}+\|\phi_{\varepsilon,\xi}\|_{H}^{2}\right)\le\varepsilon^{\frac{9-n}{2}}=o(\varepsilon).
\end{align*}
since $ $ $\|\phi_{\varepsilon,\xi}\|_{H}\le\varepsilon^{1/2}\|\phi_{\varepsilon,\xi}\|_{\varepsilon}\le\varepsilon^{5/2}$
by Proposition \ref{prop:phieps}. 

For $I_{2}$ we have, by (\ref{eq:phiweps-2}) and Remark \ref{remark:Weps}
in a similar way we get 
\begin{align*}
I_{2}\le & \frac{1}{\varepsilon^{n}}\left(\ensuremath{\int\limits _{M}}\psi^{2}\left(W_{\varepsilon,\xi}+\phi_{\varepsilon,\xi}\right)d\mu_{g}\right)^{\frac{1}{2}}\left(\ensuremath{\int\limits _{M}}\phi_{\varepsilon,\xi}^{4}d\mu_{g}\right)^{\frac{1}{2}}\\
 & +\frac{1}{\varepsilon^{n}}\left(\ensuremath{\int\limits _{M}}\psi^{3}\left(W_{\varepsilon,\xi}+\phi_{\varepsilon,\xi}\right)d\mu_{g}\right)^{\frac{1}{3}}\left(\ensuremath{\int\limits _{M}}\phi_{\varepsilon,\xi}^{3}d\mu_{g}\right)^{\frac{1}{3}}\left(\ensuremath{\int\limits _{M}}W_{\varepsilon,\xi}^{3}d\mu_{g}\right)^{\frac{1}{3}}\\
\le & \frac{1}{\varepsilon^{n}}\left\Vert \psi\left(W_{\varepsilon,\xi}+\phi_{\varepsilon,\xi}\right)\right\Vert _{H}\|\phi_{\varepsilon,\xi}\|_{H}^{2}+\\
 & +\frac{\varepsilon^{\frac{n}{3}}}{\varepsilon^{n}}\left\Vert \psi\left(W_{\varepsilon,\xi}+\phi_{\varepsilon,\xi}\right)\right\Vert _{H}\|\phi_{\varepsilon,\xi}\|_{H}\left|W_{\varepsilon,\xi}\right|_{\varepsilon,3}\\
\le & \varepsilon^{-n+\frac{n+2}{2}+5}+\varepsilon^{-\frac{2}{3}n+\frac{n+2}{2}+\frac{5}{2}}=\varepsilon^{\frac{12-n}{2}}+\varepsilon^{\frac{21-n}{6}}=o(\varepsilon)
\end{align*}
since $n=3,4$. Then (\ref{new1}) follows.

Let us prove (\ref{new2}). Since $0\le\psi\le1/q$ we have

\begin{multline*}
\left[G'_{\varepsilon}\left(W_{\varepsilon,\xi_{0}}+\phi_{\varepsilon,\xi_{0}}\right)-G'_{\varepsilon}\left(W_{\varepsilon,\xi_{0}}\right)\right]\left[\left(\frac{\partial}{\partial y_{h}}W_{\varepsilon,\xi(y)}\right)_{|_{y=0}}\right]\\
\le\left|\frac{c}{\varepsilon^{n}}\int_{M}\left\{ \psi\left(W_{\varepsilon,\xi}+\phi_{\varepsilon,\xi}\right)-\psi\left(W_{\varepsilon,\xi}\right)\right\} W_{\varepsilon,\xi_{0}}\left(\frac{\partial}{\partial y_{h}}W_{\varepsilon,\xi(y)}\right)_{|_{y=0}}\right|\\
+\left|\frac{c}{\varepsilon^{n}}\int_{M}\left\{ \psi^{2}\left(W_{\varepsilon,\xi}+\phi_{\varepsilon,\xi}\right)-\psi^{2}\left(W_{\varepsilon,\xi}\right)\right\} W_{\varepsilon,\xi_{0}}\left(\frac{\partial}{\partial y_{h}}W_{\varepsilon,\xi(y)}\right)_{|_{y=0}}\right|\\
+\left|\frac{c}{\varepsilon^{n}}\int_{M}\psi\left(W_{\varepsilon,\xi}+\phi_{\varepsilon,\xi}\right)\phi_{\varepsilon,\xi_{0}}\left(\frac{\partial}{\partial y_{h}}W_{\varepsilon,\xi(y)}\right)_{|_{y=0}}\right|\\
+\left|\frac{c}{\varepsilon^{n}}\int_{M}\psi^{2}\left(W_{\varepsilon,\xi}+\phi_{\varepsilon,\xi}\right)\phi_{\varepsilon,\xi_{0}}\left(\frac{\partial}{\partial y_{h}}W_{\varepsilon,\xi(y)}\right)_{|_{y=0}}\right|\\
\le\left|\frac{c}{\varepsilon^{n}}\int_{M}\left\{ \psi\left(W_{\varepsilon,\xi}+\phi_{\varepsilon,\xi}\right)-\psi\left(W_{\varepsilon,\xi}\right)\right\} W_{\varepsilon,\xi_{0}}\left(\frac{\partial}{\partial y_{h}}W_{\varepsilon,\xi(y)}\right)_{|_{y=0}}\right|\\
+\left|\frac{c}{\varepsilon^{n}}\int_{M}\psi\left(W_{\varepsilon,\xi}+\phi_{\varepsilon,\xi}\right)\phi_{\varepsilon,\xi_{0}}\left(\frac{\partial}{\partial y_{h}}W_{\varepsilon,\xi(y)}\right)_{|_{y=0}}\right|\\
\le\left|\frac{c}{\varepsilon^{n}}\int_{M}\left\{ \psi'\left(W_{\varepsilon,\xi}+\theta\phi_{\varepsilon,\xi}\right)\left[\phi_{\varepsilon,\xi}\right]\right\} W_{\varepsilon,\xi_{0}}\left(\frac{\partial}{\partial y_{h}}W_{\varepsilon,\xi(y)}\right)_{|_{y=0}}\right|\\
+\left|\frac{c}{\varepsilon^{n}}\int_{M}\psi\left(W_{\varepsilon,\xi}+\phi_{\varepsilon,\xi}\right)\phi_{\varepsilon,\xi_{0}}\left(\frac{\partial}{\partial y_{h}}W_{\varepsilon,\xi(y)}\right)_{|_{y=0}}\right|:=D_{1}+D_{2}
\end{multline*}
for some $0<\theta<1$. 

By Lemma \ref{lem:psiprimo}, Remark \ref{remark:Weps}, recalling
that $\|\phi_{\varepsilon,\xi}\|_{H}\le\varepsilon^{1/2}\|\phi_{\varepsilon,\xi}\|_{H}\le\varepsilon^{5/2}$
and that $\left\Vert \frac{\partial}{\partial y_{h}}W_{\varepsilon,\xi(y)}\right\Vert _{\varepsilon}=O\left(\frac{1}{\varepsilon}\right)$
(cfr. eq (\ref{eq:lemma61})) we have
\begin{align*}
D_{1}\le & \frac{c}{\varepsilon^{n}}\left(\int_{M}\left\{ \psi'\left(W_{\varepsilon,\xi}+\theta\phi_{\varepsilon,\xi}\right)\left[\phi_{\varepsilon,\xi}\right]\right\} ^{3}\right)^{\frac{1}{3}}\left(\int_{M}W_{\varepsilon,\xi(y)}^{3}\right)^{\frac{1}{3}}\left(\int_{M}\left(\frac{\partial}{\partial y_{h}}W_{\varepsilon,\xi(y)}\right)^{3}\right)^{\frac{1}{3}}\\
\le & c\frac{\varepsilon^{\frac{2}{3}n}}{\varepsilon^{n}}\left\Vert \psi'\left(W_{\varepsilon,\xi}+\theta\phi_{\varepsilon,\xi}\right)\left[\phi_{\varepsilon,\xi}\right]\right\Vert _{H}\|W_{\varepsilon,\xi(y)}\|_{\varepsilon}\left\Vert \frac{\partial}{\partial y_{h}}W_{\varepsilon,\xi(y)}\right\Vert _{\varepsilon}\\
\le & c\varepsilon^{-1-\frac{n}{3}}\left\Vert \psi'\left(W_{\varepsilon,\xi}+\theta\phi_{\varepsilon,\xi}\right)\left[\phi_{\varepsilon,\xi}\right]\right\Vert _{H}\le c\varepsilon^{-1-\frac{n}{3}}\|\phi_{\varepsilon,\xi}\|_{H}\left\{ \varepsilon^{2}+\|\phi_{\varepsilon,\xi}\|_{H}\right\} \\
\le & c\varepsilon^{-1-\frac{n}{3}}\varepsilon^{\frac{5}{2}}\varepsilon^{2}=c\varepsilon^{\frac{7}{2}-\frac{n}{3}}=o(\varepsilon).
\end{align*}
In a similar way, using (\ref{eq:phiweps-2}) and the above estimates
we get
\begin{align*}
D_{2}\le & \frac{c}{\varepsilon^{n}}\left(\int_{M}\psi^{3}\left(W_{\varepsilon,\xi}+\phi_{\varepsilon,\xi}\right)\right)^{\frac{1}{3}}\left(\int_{M}\phi_{\varepsilon,\xi_{0}}^{3}\right)^{\frac{1}{3}}\left(\int_{M}\left(\frac{\partial}{\partial y_{h}}W_{\varepsilon,\xi(y)}\right)^{3}\right)^{\frac{1}{3}}\\
\le & c\frac{\varepsilon^{\frac{n}{3}}}{\varepsilon^{n}}\left\Vert \psi\left(W_{\varepsilon,\xi}+\phi_{\varepsilon,\xi}\right)\right\Vert _{H}\left\Vert \phi_{\varepsilon,\xi}\right\Vert _{H}\left\Vert \frac{\partial}{\partial y_{h}}W_{\varepsilon,\xi(y)}\right\Vert _{\varepsilon}\\
\le & c\varepsilon^{-\frac{2}{3}n-1}\varepsilon^{\frac{5}{2}}\left\Vert \psi\left(W_{\varepsilon,\xi}+\phi_{\varepsilon,\xi}\right)\right\Vert _{H}\le c\varepsilon^{-\frac{2}{3}n+\frac{3}{2}}\varepsilon^{\frac{n+2}{n}}\left(1+\|\phi_{\varepsilon,\xi}\|_{\varepsilon}\right)\\
\le & c\varepsilon^{\frac{15-n}{6}}=o(\varepsilon)
\end{align*}
and (\ref{new2}) is proved.

The prove of (\ref{new3}) requires to estimate that 
\begin{equation}
I:=\left|\frac{1}{\varepsilon^{n}}\int\limits _{M}\left[q^{2}\psi^{2}\left(\ensuremath{W_{\varepsilon,\xi(y)}+\phi_{\varepsilon,\xi(y)}}\right)-2q\psi\left(\ensuremath{W_{\varepsilon,\xi(y)}+\phi_{\varepsilon,\xi(y)}}\right)\right]\left(\ensuremath{W_{\varepsilon,\xi(y)}+\phi_{\varepsilon,\xi(y)}}\right)Z_{\varepsilon,\xi(y)}^{l}\right|=o(\varepsilon),\label{new4}
\end{equation}
where the functions $Z_{\varepsilon,\xi(y)}^{l}$ are defined in (\ref{eq:Zi}).
By (\ref{new4}) it is possible to complete the proof the lemma, with
the same arguments the proof of (5.10) in \cite{MP09}, which we refer
to for the missing details. To prove (\ref{new4}), since $0<\psi<1/q$,
we get, as before, 
\begin{align*}
I\le & \left|\frac{c}{\varepsilon^{n}}\int\limits _{M}\psi\left(\ensuremath{W_{\varepsilon,\xi(y)}+\phi_{\varepsilon,\xi(y)}}\right)\left(\ensuremath{W_{\varepsilon,\xi(y)}+\phi_{\varepsilon,\xi(y)}}\right)Z_{\varepsilon,\xi(y)}^{l}\right|\le\\
\le & c\frac{\varepsilon^{\frac{n+2}{2}}}{\varepsilon^{n}}\left(\int\limits _{M}\psi^{2^{*}}\left(\ensuremath{W_{\varepsilon,\xi(y)}+\phi_{\varepsilon,\xi(y)}}\right)\right)^{\frac{1}{2*}}\left(\frac{1}{\varepsilon^{n}}\int\limits _{M}\left(\ensuremath{W_{\varepsilon,\xi(y)}+\phi_{\varepsilon,\xi(y)}}\right)^{\frac{4n}{n+2}}\right)^{\frac{n+2}{4n}}\\
 & \times\left(\frac{1}{\varepsilon^{n}}\int\limits _{M}\left(Z_{\varepsilon,\xi(y)}^{l}\right)^{\frac{4n}{n+2}}\right)^{\frac{n+2}{4n}}\\
\le & c\varepsilon^{\frac{2-n}{2}}\left\Vert \psi\left(\ensuremath{W_{\varepsilon,\xi(y)}+\phi_{\varepsilon,\xi(y)}}\right)\right\Vert _{H}\left|W_{\varepsilon,\xi(y)}+\phi_{\varepsilon,\xi(y)}\right|_{\varepsilon,\frac{4n}{n+2}}\left|Z_{\varepsilon,\xi(y)}^{l}\right|_{\varepsilon,\frac{4n}{n+2}}.
\end{align*}
Arguing as in Remark \ref{remark:Weps}, we have that $\left|Z_{\varepsilon,\xi(y)}^{l}\right|_{\varepsilon,\frac{4n}{n+2}}\rightarrow|\varphi^{l}|_{\frac{4n}{n+2}},$
so, by (\ref{eq:phiweps-2}) we obtain 
\[
I\le c\varepsilon^{\frac{2-n}{2}}\varepsilon^{\frac{n+2}{2}}=c\varepsilon^{2}.
\]
This concludes the proof.\end{proof}
\begin{lem}
\label{lem:Geps2}It holds that 
\[
G_{\varepsilon}(W_{\varepsilon,\xi}):=\frac{1}{\varepsilon^{n}}\int\limits _{M}\psi(W_{\varepsilon,\xi})W_{\varepsilon,\xi}^{2}d\mu_{g}=o(\varepsilon)
\]
$C^{1}-$uniformly with respect to $\xi\in M$ as $\varepsilon$ goes
to zero. \end{lem}
\begin{proof}
At first we have, by Remark \ref{remark:Weps} and by (\ref{eq:phiweps-1})

\[
G_{\varepsilon}(W_{\varepsilon,\xi})\le c\frac{1}{\varepsilon^{n}}\left(\int\limits _{M}\psi^{3}(W_{\varepsilon,\xi})\right)^{\frac{1}{3}}\left(\int\limits _{M}W_{\varepsilon,\xi}^{3}\right)^{\frac{2}{3}}\le c\frac{1}{\varepsilon^{n}}\varepsilon^{\frac{n+2}{2}}\varepsilon^{\frac{2}{3}n}=c\varepsilon^{\frac{n}{6}+1}=o(\varepsilon).
\]
We want now to prove the $C^{1}$ convergence, id est, if $\xi(y)=\exp_{\xi}(y)$
for $y\in B(0,r),$ we will prove that 
\begin{equation*}
\left.\frac{\partial}{\partial y_{h}}G_{\varepsilon}(W_{\varepsilon,\xi})\right|_{y=0}  =  \frac{2}{\varepsilon^{n}}\int_{M}\left(2q\psi(W_{\varepsilon,\xi})-q^{2}\psi^{2}(W_{\varepsilon,\xi})\right)W_{\varepsilon,\xi}\left[\left.\frac{\partial}{\partial y_{h}}W_{\varepsilon,\xi(h)}\right|_{y=0}\right]d\mu_{g}
\end{equation*}
for $h=1,\dots,n-1$. Since $0<\psi<1/q$, immediately we have 
\[
\left|\left.\frac{\partial}{\partial y_{h}}G_{\varepsilon}(W_{\varepsilon,\xi})\right|_{y=0}\right|\le c\left|\frac{1}{\varepsilon^{n}}\int\limits _{M}\psi(W_{\varepsilon,\xi(y)})W_{\varepsilon,\xi(h)}\left.\frac{\partial}{\partial y_{h}}W_{\varepsilon,\xi(h)}\right|_{y=0}d\mu_{g}\right|
\]
Set $I_{1}$ the quantity inside the absolute value at the r.h.s.
of the above equation. Using the Fermi coordinates and the previous
estimates we get 
\begin{multline*}
\frac{1}{\varepsilon^{2}}I_{1}(\varepsilon,\xi)=\int\limits _{\mathbb{R}_{+}^{n}}\frac{\tilde{v}_{\varepsilon,\xi}(z)}{\varepsilon^{2}}2U(z)\chi_{R}(\varepsilon z)|g_{\xi}(\varepsilon z)|^{1/2}\times\\
\times\left\{ \sum_{k=1}^{3}\left[\frac{1}{\varepsilon}\frac{\partial U(z)}{\partial z_{k}}\chi_{R}(\varepsilon z)+U(z)\frac{\partial\chi_{r}(\varepsilon z)}{\partial z_{k}}\right]\frac{\partial}{\partial y_{h}}\mathcal{H}_{k}(0,\exp_{\xi}(\varepsilon z))\right\} dz.
\end{multline*}
where $\mathcal{H}_{k}(x,y)$ is introduced in Definition (\ref{def:Estorto}).
Since $|g_{\xi}(\varepsilon z)|^{1/2}=1+O(\varepsilon|z|)$ and by
Lemma \ref{lem:Estorto2} we have
\begin{align*}
I_{1}(\varepsilon,\xi)= & 2\varepsilon\int\limits _{\mathbb{R}_{+}^{n}}\tilde{v}_{\varepsilon,\xi}(z)U(z)\frac{\partial U(z)}{\partial z_{h}}\chi_{R}^{2}(\varepsilon z)dz+o(\varepsilon)\\
= & 2\varepsilon\int\limits _{\mathbb{R}_{+}^{n}}\tilde{v}_{\varepsilon,\xi}(z)U(z)\frac{\partial U(z)}{\partial z_{h}}dz+o(\varepsilon)
\end{align*}
By Lemma \ref{lem:e5} we have that ${\displaystyle \left\{ \frac{1}{\varepsilon_{n}^{2}}\tilde{v}_{\varepsilon_{n},\xi}\right\} _{n}}$
converges to $\gamma$ weakly in $L^{2^{*}}(\mathbb{R}_{+}^{n})$,
so we have 
\[
I_{1}(\varepsilon,\xi)=2\varepsilon\int\limits _{\mathbb{R}^{n}}\gamma U(z)\frac{\partial U(z)}{\partial z_{h}}dz+o(\varepsilon)
\]
where $h=1,\dots,n-1$. Finally, we have that ${\displaystyle \int\limits _{\mathbb{R}^{n}}\gamma(z)U(z)\frac{\partial U(z)}{\partial z_{h}}dz=0}$
because both $\gamma$ (see Remark \ref{rem:gammasym}) and $U$ are
symmetric with respect $z_{1},\dots,z_{n-1}$ while $\frac{\partial U(z)}{\partial z_{h}}$
is antisymmetric. This concludes the proof.
\end{proof}
We can now prove Lemma \ref{lem:Iexp}.
\begin{proof}[Proof of Lemma \ref{lem:Iexp}]
 We want to estimate 
\[
I_{\varepsilon}(W_{\varepsilon,\xi}+\phi_{\varepsilon,\xi})=J_{\varepsilon}(W_{\varepsilon,\xi}+\phi_{\varepsilon,\xi})+\frac{\omega^{2}}{2}G_{\varepsilon}(W_{\varepsilon,\xi}+\phi_{\varepsilon,\xi}),
\]
By Remark \ref{lem:Jeps} we have that 
\[
J_{\varepsilon}(W_{\varepsilon,\xi}+\phi_{\varepsilon,\xi})=J_{\varepsilon}(W_{\varepsilon,\xi})+o(\varepsilon)=C-\varepsilon\alpha H(\xi)+o(\varepsilon)
\]
$C^{1}$ uniformly with respect to $\xi\in\partial M$ as $\varepsilon$
goes to zero. Moreover by Lemma \ref{lem:Geps1} and by Lemma \ref{lem:Geps2}
we have that 
\[
G_{\varepsilon}(W_{\varepsilon,\xi}+\phi_{\varepsilon,\xi})=o(\varepsilon)
\]
$C^{1}$ uniformly with respect to $\xi\in\partial M$ and this concludes
the proof. 
\end{proof}

\subsection{Sketch of the proof of Theorem \ref{thm:main}}

In section \ref{sec:reduction}, Proposition \ref{prop:phieps} we
found a function $\phi_{\varepsilon,\xi}$ solving (\ref{eq:red1}).
By Lemma \ref{lem:sol2} we can solve (\ref{eq:red2}) once we have
a critical point of functional $\tilde{I}_{\varepsilon}$. At this
point by Lemma \ref{lem:Iexp} and by definition of $C^{1}$ stable
critical point (Def. \ref{def:stable}) we can complete the proof.

\appendix

\section{\label{app}Technical lemmas}
\begin{lem}
There exists $\varepsilon_{0}>0$ and $c>0$ such that, for any $\xi_{0}\in\partial M$
and for any $\varepsilon\in(0,\varepsilon_{0})$ it holds
\begin{equation}
\left\Vert \frac{\partial}{\partial y_{h}}Z_{\varepsilon,\xi(y)}^{l}\right\Vert _{\varepsilon}=O\left(\frac{1}{\varepsilon}\right),\ \ \left\Vert \frac{\partial}{\partial y_{h}}W_{\varepsilon,\xi(y)}\right\Vert _{\varepsilon}=O\left(\frac{1}{\varepsilon}\right),\label{eq:lemma61}
\end{equation}
for $h=1,\dots,n-1$, $l=1,\dots,n$
\end{lem}

\begin{lem}
\label{lem:e5}Let us consider the functions 
\[
\tilde{v}_{\varepsilon,\xi}(z)=\left\{ \begin{array}{cl}
\psi(W_{\varepsilon,\xi})\left(\Psi_{\xi}^{\partial}(\varepsilon z)\right) & \text{ for }z\in D^{+}(R/\varepsilon)\\
\\
0 & \text{ for }z\in\mathbb{R}^{3}\smallsetminus D^{+}(R/\varepsilon)
\end{array}\right.
\]
Where $D^{+}(r/\varepsilon)=\left\{ z=(\bar{z},z_{n}),\ \bar{z}\in\mathbb{R}^{n-1},|\bar{z}|<r/\varepsilon,\ 0\le z_{n}<R/\varepsilon)\right\} $.
Then there exists a constant $c>0$ such that 
\[
\|\tilde{v}_{\varepsilon,\xi}(z)\|_{L^{2^{*}}(\mathbb{R}_{+}^{n})}\le c\varepsilon^{2}.
\]
Furthermore, take a sequence $\varepsilon_{n}\rightarrow0$, up to
subsequences, ${\displaystyle \left\{ \frac{1}{\varepsilon_{n}^{2}}\tilde{v}_{\varepsilon_{n},\xi}\right\} _{n}}$
converges weakly in $L^{2^{*}}(\mathbb{R}_{+}^{n})$ as $\varepsilon$
goes to $0$ to a function $\gamma\in D^{1,2}(\mathbb{R}^{3})$. The
function $\gamma$ solves, in a weak sense, the equation 
\begin{equation}
-\Delta\gamma=qU^{2}\text{ in }\mathbb{R}_{+}^{n}\label{eq:egamma-sol}
\end{equation}
\end{lem}
\begin{proof}
We prove the Lemma for Problem (\ref{eq:Pnn}), being the Problem
(\ref{eq:Pdn}) completely analogous. By definition of $\tilde{v}_{\varepsilon,\xi}(z)$
and by (\ref{eq:Pnn}) we have, for all $z\in D^{+}(r/\varepsilon)$,
\begin{multline}
-\sum_{ij}\partial_{j}\left(|g_{\xi}(\varepsilon z)|^{1/2}g_{\xi}^{ij}(\varepsilon z)\partial_{i}\tilde{v}_{\varepsilon,\xi}(z)\right)=\\
=\varepsilon^{2}|g_{\xi}(\varepsilon z)|^{1/2}\left\{ qU^{2}(z)\chi_{r}^{2}(\varepsilon z)-\left[1+q^{2}U^{2}(z)\chi_{R}^{2}(\varepsilon z)\right]\tilde{v}_{\varepsilon,\xi}(z)\right\} \label{eq:e11-1}
\end{multline}
By (\ref{eq:e11-1}), and remarking that $\tilde{v}_{\varepsilon,\xi}(z)\ge0$
we have 
\begin{align*}
 & \|\tilde{v}_{\varepsilon,\xi}(z)\|_{D^{1,2}\left(D^{+}(r/\varepsilon)\right)}^{2}\le C\int\limits _{D^{+}(R/\varepsilon)}|g_{\xi}(\varepsilon z)|^{1/2}g_{\xi}^{ij}(\varepsilon z)\partial_{i}\tilde{v}_{\varepsilon,\xi}(z)\partial_{j}\tilde{v}_{\varepsilon,\xi}(z)dz\\
 & =C\varepsilon^{2}\int\limits _{D^{+}(R/\varepsilon)}|g_{\xi}(\varepsilon z)|^{1/2}\left\{ qU^{2}(z)\chi_{R}^{2}(\varepsilon z)\tilde{v}_{\varepsilon,\xi}(z)-\left[1+q^{2}U^{2}(z)\chi_{R}^{2}(\varepsilon z)\right]\tilde{v}_{\varepsilon,\xi}^{2}(z)\right\} dz\\
 & \le C\varepsilon^{2}\int\limits _{D^{+}(R/\varepsilon)}|g_{\xi}(\varepsilon z)|^{1/2}qU^{2}(z)\chi_{R}^{2}(\varepsilon|z|)\tilde{v}_{\varepsilon,\xi}(z)dz\\
 & \le C\varepsilon^{2}\|\tilde{v}_{\varepsilon,\xi}(z)\|_{L^{2^{*}}\left(D^{+}(R/\varepsilon)\right)}\|U\|_{L^{\frac{4n}{n+2}}}^{2}\le C\varepsilon^{2}\|\tilde{v}_{\varepsilon,\xi}(z)\|_{D^{1,2}\left(D^{+}(R/\varepsilon)\right)}
\end{align*}
Thus we have 
\begin{equation}
\|\tilde{v}_{\varepsilon,\xi}(z)\|_{D^{1,2}\left(D^{+}(R/\varepsilon)\right)}\le C\varepsilon^{2}\text{ and }|\tilde{v}_{\varepsilon,\xi}(z)|_{L^{2^{*}}(\mathbb{R}_{+}^{n})}\le C\varepsilon^{2}.\label{eq:e13}
\end{equation}
By (\ref{eq:e13}), if $\varepsilon_{n}$ is a sequence which goes
to zero, the sequence $\left\{ \frac{1}{\varepsilon_{n}^{2}}\tilde{v}_{\varepsilon_{n},\xi}\right\} _{n}$
is bounded in $L^{2^{*}}(\mathbb{R}_{+}^{n})$. Then, up to subsequence,
$\left\{ \frac{1}{\varepsilon_{n}^{2}}\tilde{v}_{\varepsilon_{n},\xi}\right\} _{n}$
converges to some $\tilde{\gamma}\in L^{2^{*}}(\mathbb{R}_{+}^{n})$
weakly in $L^{2^{*}}(\mathbb{R}_{+}^{n})$. 

Moreover, by (\ref{eq:e11-1}), for any $\varphi\in C_{0}^{\infty}(\mathbb{R}_{+}^{n})$,
it holds 
\begin{multline}
\int\limits _{\text{supp }\varphi}\sum_{ij}|g_{\xi}(\varepsilon z)|^{1/2}g_{\xi}^{ij}(\varepsilon z)\partial_{i}\frac{\tilde{v}_{\varepsilon,\xi}(z)}{\varepsilon_{n}^{2}}\partial_{j}\varphi(z)dz=\\
\int\limits _{\text{supp }\varphi}\left\{ qU^{2}(z)\chi_{r}^{2}(\varepsilon|z|)-\left[1+q^{2}U^{2}(z)\chi_{R}^{2}(\varepsilon z)\right]\tilde{v}_{\varepsilon,\xi}(z)\right\} |g_{\xi}(\varepsilon z)|^{1/2}\varphi(z)dz.\label{eq:213bis}
\end{multline}
Consider now the functions 
\[
v_{\varepsilon,\xi}(z):=\psi(W_{\varepsilon,\xi})\left(\Psi_{\xi}^{\partial}(\varepsilon z)\right)\chi_{R}(\varepsilon z)=\tilde{v}_{\varepsilon,\xi}(z)\chi_{r}(\varepsilon z)\text{ for }z\in\mathbb{R}_{+}^{n}.
\]
We have immediately that $v_{\varepsilon,\xi}(z)$ is bounded in $D^{1,2}(\mathbb{R}_{+}^{n})$,
thus the sequence $\left\{ \frac{1}{\varepsilon_{n}^{2}}v_{\varepsilon_{n},\xi}\right\} _{n}$
converges to some $\gamma\in D^{1,2}(\mathbb{R}^{3})$ weakly in $D^{1,2}(\mathbb{R}_{+}^{n})$
and in $L^{2^{*}}(\mathbb{R}_{+}^{n})$. Finally, for any compact
set $K\subset\mathbb{R}_{+}^{n}$ eventually $v_{\varepsilon_{n},\xi}\equiv\tilde{v}_{\varepsilon_{n},\xi}$
on $K$. So it is easy to see that $\tilde{\gamma}=\gamma$.

We recall that $|g_{\xi}(\varepsilon z)|^{1/2}=1+O(\varepsilon|z|)$
and $g_{\xi}^{ij}(\varepsilon z)=\delta_{ij}+O(\varepsilon|z|)$ so,
by the weak convergence of $\left\{ \frac{1}{\varepsilon_{n}^{2}}v_{\varepsilon_{n},\xi}\right\} _{n}$
in $D^{1,2}(\mathbb{R}_{+}^{n})$, for any $\varphi\in C_{0}^{\infty}(\mathbb{R}_{+}^{n})$
we get 
\begin{multline}
\int\limits _{\text{supp }\varphi}\sum_{ij}|g_{\xi}(\varepsilon_{n}z)|^{1/2}g_{\xi}^{ij}(\varepsilon_{n}z)\partial_{i}\frac{\tilde{v}_{\varepsilon_{n},\xi}(z)}{\varepsilon_{n}^{2}}\partial_{j}\varphi(z)dz\\
=\int\limits _{\text{supp }\varphi}\sum_{ij}|g_{\xi}(\varepsilon_{n}z)|^{1/2}g_{\xi}^{ij}(\varepsilon_{n}z)\partial_{i}\frac{v_{\varepsilon_{n},\xi}(z)}{\varepsilon_{n}^{2}}\partial_{j}\varphi(z)dz\\
\rightarrow\int\limits _{\mathbb{R}^{3}}\sum_{i}\partial_{i}\gamma(z)\partial_{i}\varphi(z)dz\text{ as }n\rightarrow\infty.\label{eq:e14bis}
\end{multline}
Thus by (\ref{eq:213bis}) and by (\ref{eq:e14bis}) and because $\left\{ \frac{1}{\varepsilon_{n}^{2}}\tilde{v}_{\varepsilon_{n},\xi}\right\} _{n}$
converges to $\gamma$ weakly in $L^{2^{*}}(\mathbb{R}_{+}^{n})$
we get 
\[
\int\limits _{\mathbb{R}_{+}^{n}}\sum_{i}\partial_{i}\gamma(z)\partial_{i}\varphi(z)dz=q\int\limits _{\mathbb{R}_{+}^{n}}U^{2}(z)\varphi(z)dz\text{ for all }\varphi\in C_{0}^{\infty}(\mathbb{R}_{+}^{n}).
\]
So, finally, up to subsequences, ${\displaystyle \left\{ \frac{1}{\varepsilon_{n}^{2}}\tilde{v}_{\varepsilon_{n},\xi}\right\} _{n}}$
converges to $\gamma$, weakly in $L^{2^{*}}(\mathbb{R}_{+}^{n})$
and the function $\gamma\in D^{1,2}(\mathbb{R}_{+}^{n})$ is a weak
solution of $-\Delta\gamma=qU^{2}$ in $\mathbb{R}_{+}^{n}$.\end{proof}
\begin{rem}
\label{rem:gammasym}We remark that $\gamma$ is positive and decays
exponentially at infinity with its first derivative because it solves
${\displaystyle -\Delta\gamma=qU^{2}}$ in $\mathbb{R}_{+}^{n}$.
Moreover its is symmetric with respect to the first $n-1$ variables. \end{rem}
\begin{defn}
\label{def:Estorto}Let $\xi_{0}\in\partial M$. We introduce the
functions ${\mathcal E}$ and $\tilde{\mathcal{E}}$ as follows.

\[
\mathcal{E}(y,x)=\left(\exp_{\xi(y)}^{\partial}\right)^{-1}(x)=\left(\exp_{\exp_{\xi_{0}}^{\partial}y}^{\partial}\right)^{-1}(\exp_{\xi_{0}}^{\partial}\bar{\eta})=\tilde{\mathcal{E}}(y,\bar{\eta})
\]
where $x,\xi(y)\in\partial M$, $y,\bar{\eta}\in B(0,R)\subset\mathbb{R}^{n-1}$
and $\xi(y)=\exp_{\xi_{0}}^{\partial}y$, $x=\exp_{\xi_{0}}^{\partial}\bar{\eta}$.
$ $Using Fermi coordinates, in a similar way we define
\[
\mathcal{H}(y,x)=\left(\psi_{\xi(y)}^{\partial}\right)^{-1}(x)=\left(\psi_{\exp_{\xi_{0}}^{\partial}y}^{\partial}\right)^{-1}\left(\psi_{\xi_{0}}^{\partial}(\bar{\eta},\eta_{n})\right)=\tilde{\mathcal{H}}(y,\bar{\eta},\eta_{n})=(\tilde{\mathcal{E}}(y,\bar{\eta}),\eta_{n})
\]
where $x\in M$, $\eta=(\bar{\eta},\eta_{n})$, with $\bar{\eta}\in B(0,R)\subset\mathbb{R}^{n-1}$
and $0\le\eta_{n}<R$, $\xi(y)=\exp_{\xi_{0}}^{\partial}y\in\partial M$
and $x=\psi_{\xi_{0}}^{\partial}(\eta)$. \end{defn}
\begin{lem}
\label{lem:Estorto}It holds
\[
\frac{\partial\tilde{\mathcal{E}}_{k}}{\partial y_{j}}(0,0)=-\delta_{jk}\text{ for }j,k=1,\dots,n-1
\]
\end{lem}
\begin{proof}
We recall that $\tilde{\mathcal{E}}(y,\bar{\eta})=\left(\exp_{\xi(y)}^{\partial}\right)^{-1}(\exp_{\xi_{0}}^{\partial}\bar{\eta})$.
Let us introduce, for $y,\bar{\eta}\in B(0,R)\subset\mathbb{R}^{n-1}$
\begin{align*}
F(y,\bar{\eta}) & =\left(\exp_{\xi_{0}}^{\partial}\right)^{-1}\left(\exp_{\xi(y)}^{\partial}(\bar{\eta})\right)\\
\Gamma(y,\bar{\eta}) & =\left(y,F(y,\bar{\eta})\right).
\end{align*}
We notice that $\Gamma^{-1}=(y,\tilde{\mathcal{E}}(y,\bar{\eta}))$.
We can easily compute the derivative of $\Gamma$. We have 
\[
\Gamma'(\hat{y},\hat{\eta})[\tilde{y},\tilde{\eta}]=\left(\begin{array}{cc}
\text{Id}_{\mathbb{R}^{n-1}} & 0\\
F_{y}'(\hat{y},\hat{\eta}) & F_{\eta}'(\hat{y},\hat{\eta})
\end{array}\right)\left(\begin{array}{c}
\tilde{y}\\
\tilde{\eta}
\end{array}\right),
\]
thus
\[
\left(\Gamma^{-1}\right)^{\prime}(\hat{y},\hat{\eta})[\tilde{y},\tilde{\eta}]=\left(\begin{array}{cc}
\text{Id}_{\mathbb{R}^{n-1}} & 0\\
-\left(F_{\eta}'(\hat{y},\hat{\eta})\right)^{-1}F_{y}'(\hat{y},\hat{\eta}) & \left(F_{\eta}'(\hat{y},\hat{\eta})\right)^{-1}
\end{array}\right)\left(\begin{array}{c}
\tilde{y}\\
\tilde{\eta}
\end{array}\right)
\]
Now, by direct computation we have that
\[
F_{\eta}'(0,\hat{\eta})=\text{Id}_{\mathbb{R}^{n-1}}\text{ and }F_{y}'(\hat{y},0)=\text{Id}_{\mathbb{R}^{n-1}},
\]
so $\frac{\partial\tilde{\mathcal{E}}_{k}}{\partial y_{j}}(0,0)
=\left(-\left(F_{\eta}'(0,0)\right)^{-1}F_{y}'(0,0)\right)_{jk}=-\delta_{jk}$. \end{proof}
\begin{lem}
\label{lem:Estorto2}We have that 
\begin{align*}
\tilde{\mathcal{H}}(0,\bar{\eta},\eta_{n})
= & (\bar{\eta},\eta_{n})\text{ for }\bar{\eta}\in\mathbb{R}^{n-1},\eta_{n}\in\mathbb{R}_{+}\\
\frac{\partial\tilde{\mathcal{H}}_{k}}{\partial y_{j}}(0,0,\eta_{n})
= & -\delta_{jk}\text{ for }j,k=1,\dots,n-1,\eta_{n}\in\mathbb{R}_{+}\\
\frac{\partial\tilde{\mathcal{H}}_{n}}{\partial y_{j}}(y,\bar{\eta},\eta_{n})= & 0\text{ for }j
=1,\dots,n-1,y,\bar{\eta}\in\mathbb{R}^{n-1},\eta_{n}\in\mathbb{R}_{+}
\end{align*}
\end{lem}
\begin{proof}
The first two claim follows immediately by Definition \ref{def:Estorto}
and Lemma \ref{lem:Estorto}. For the last claim, observe that $\tilde{\mathcal{H}}_{k}(y,\bar{\eta},\eta_{n})
=\tilde{\mathcal{E}}_{k}(y,\bar{\eta})$
which does not depends on $\eta_{n}$ as well as its derivatives.
\end{proof}

\end{document}